\documentclass[11pt]{amsart}
\usepackage{amsaddr}
\usepackage{graphicx,epsfig}

\usepackage{mathtools}
\usepackage{enumerate}

\usepackage{color}
\usepackage{verbatim}

 \usepackage[bookmarks=true,
bookmarksnumbered=true, breaklinks=true,
pdfstartview=FitH, hyperfigures=false,
plainpages=false, naturalnames=true,
colorlinks=false,
pdfpagelabels]{hyperref}

\makeatletter
\newtheorem*{rep@theorem}{\rep@title}
\newcommand{\newreptheorem}[2]{%
\newenvironment{rep#1}[1]{%
 \def\rep@title{#2 \ref{##1}}%
 \begin{rep@theorem}}%
 {\end{rep@theorem}}}
\makeatother
\newreptheorem{theorem}{Theorem}
\newreptheorem{corollary}{Corollary}

\newcommand{\R}{\mathbb{R}}

\newcommand \B[1][n]{B_2^{#1}}
\newcommand{\N}{\mathbb{N}}
\newcommand{\vol}[1][n]{\operatorname{vol}_{#1}}
\newcommand{\Vol}{\operatorname{vol}_{nm}}
\def \s{\mathbb{S}^{n-1}}
\def \S{\mathbb{S}^{nm-1}}

\newcommand {\conbod}[1][n,m] {\mathcal{K}^{#1}}
\newcommand {\conbodo}[1][n,m] {\mathcal{K}^{#1}_o}
\newcommand {\conbodio}[1][n,m] {\mathcal{K}^{#1}_{(o)}}

\renewcommand{\P}[1][Q]{\Pi_{#1,p}}

\newcommand{\PP}[1][Q]{\Pi_{#1,p}^{\circ}}

\newcommand{\G}[1][Q]{\Gamma_{#1,p}}

\newcommand {\M}[1][n,m] {\operatorname{M}_{#1}(\R)}

\mathtoolsset{showonlyrefs}

\newtheorem{theorem}{Theorem}[section]
\newtheorem{definition}[theorem]{Definition}
\newtheorem{lemma}[theorem]{Lemma}
\newtheorem{proposition}[theorem]{Proposition}

\title[Higher Order Mean Zonoids]{General Higher Order $L^p$ Mean Zonoids}

\author[Langharst]{Dylan Langharst}
\address{Institut de Math\'ematiques de Jussieu \\ Sorbonne Universit\'e \\
Paris, France 75252
\\
and
\\
Department of Mathematical Sciences \\ Kent State University \\ Kent, OH 44242, USA. }
\email{dylan.langharst@imj-prg.fr}

\author[Xi]{Dongmeng Xi}
\address{Department of Mathematics \\ Shanghai University \\ Shanghai 200444, China.
\\ and \\
Institut F\"ur Diskrete Mathematik und Geometrie \\ Technische Universit\"at Wien
\\
1040 Wien, Austria}
\email{dongmeng.xi.math@gmail.com}


\begin{document}

\begin{abstract} In 1970, Schneider introduced the higher-order difference body and the associated Rogers-Shephard inequality. Recently, Haddad, Langharst, Putterman, Roysdon and Ye expanded the concept to a burgeoning higher-order Brunn-Minkowski theory. In 1991, Zhang introduced mean zonoids of a convex body, which was extended to the Firey-Brunn-Minkowski theory setting by Xi, Guo and Leng in 2014. In this note, we extend these $L^p$ mean zonoids to the higher-order setting and establish the associated isoperimetric inequality.
\end{abstract}

\maketitle

\newpage
\section{Introduction}

Zhang \cite{Zhang91} defined the mean zonoid of a convex body $K$ (compact, convex set) in $\R^n$ (the $n$-dimensional Euclidean space with standard inner product $\langle \cdot, \cdot\rangle$) via its support function: for $u\in\s$
\begin{equation}
h_{\tilde Z K}(u)=\frac{1}{\vol(K)^2}\int_K\int_K|\langle u,(x-y) \rangle|dxdy.
\label{eq:zhang}
\end{equation}
Zhang showed $\vol(\tilde Z K) \geq \vol(\tilde Z B_K)$, with equality if, and only if, $K$ is an ellipsoid, where $B_K$ is the Euclidean ball of the same volume, i.e. $B_K = \vol(\B)^{-1/n}\vol(K)^{1/n} \B$. Here, $\vol(\cdot)$ denotes the $n$-dimensional Lebesgue measure and $\B$ is the Euclidean ball (with boundary $\s$). We recall that a convex body is uniquely determined by its support function, which, in general, is given by $h_K(u)=\sup_{x\in K}\langle x,u \rangle$. That $\tilde Z K$ is indeed a convex body follows from $h_{\tilde Z K}$ defined in \eqref{eq:zhang} being a $1$-homogeneous, convex function on $\s$, and every such function is the support function of a convex body.

Xi, Guo and Leng \cite{XGL14} defined $L^p$ Mean zonoids as
\begin{equation}h_{\tilde Z_p K}(u)^p=\frac{1}{\vol(K)^2}\int_K\int_K|\langle u,(x-y) \rangle|^pdxdy.
\label{eq:xgl}
\end{equation}
Then, they showed the following relation:
$$\tilde Z_p K = \left(\frac{\vol(R_{n+p}K)}{\vol(K)}\right)^\frac{1}{p}\Gamma_p (R_{n+p} K),$$
where $\Gamma _p$ is the $L^p$ centroid body operator and $R_{n+p} K$ is the $(n+p)$th radial mean body of $K$, elaborated on below.
We remark that these were further generalized to the Orlicz-Brunn-Minkowski setting by Guo, Leng and Du \cite{GLD15}. In this work, we will generalize the $L^p$ mean zonoids to the higher-order setting. 

For our results, we identify $\R^{kl}$, with $k,l\in\N$, as $\M[k,l]$, the real vector space of $k\times l$ matrices. Firstly, we write $\conbod[k,l]$ for the set of convex bodies in $\M[k,l]$. Similarly, we write $\conbodo[k,l]$ for the collection of all such convex bodies which contain the origin and $\conbodio[k,l]$ for those that contain the origin in their interiors. Therefore, we use $v^t.x$ to denote the matrix multiplication of $v^t$ and $x$. This is, of course, an extension of the notation for the matrix multiplication of two matrices $A \in \M$ and $B \in \M[m,k]$ given by $A.B \in \M[n,k]$, and of the transpose of $A$, denoted $A^t$.
The Lebesgue measure in $\M$ is inherited from the natural identification between $\M$ and $\R^{nm}$. If $D \subset \M[n,m]$ is Lebesgue measurable, $A \in \M[n,n]$, and $ B \in \M[m,m]$, then $$\Vol(A.D) = |\det(A)|^m \Vol(D) \ \ \mathrm{and} \Vol(D.B) = |\det(B)|^n \Vol(D),$$ where $\det (A)$ is the determinant of $A$. The polar body of $G\in\conbodio[k,l]$ is
$$G^\circ=\left\{x\in\M[k,l]:h_G(x)\leq 1\right\}.$$
Notice that under these conventions, $G^\circ$ is also in $\M[k,l]$. Additionally,
\begin{equation}
\label{eq:sup_mat}
h_K (A.x) = h_{A^t. K}(x) \quad \text{and} \quad h_K (y.B) = h_{K.B^t}(y)\end{equation} for all $x\in \M[k,m], y \in \M[n,l]$, $K\in \conbod$, $A\in \M[n, k]$ and $B\in \M[l,m]$.

Next, we recall the higher-order centroid and radial mean bodies, introduced in \cite{HLPRY23,HLPRY23_2}.
\begin{definition}
  \label{d:generalcentroidbody}
  Let $p \geq 1$, $m \in \N$, and fix some $Q \in \conbodo[1,m]$. Given a compact set $L\subset \M$ with positive volume, we define the $(L^p, Q)$-centroid body of $L$, $\G L$, to be the convex body in $\M[n,1]$ with the support function 
  \begin{equation}
  h_{\G L}(v)^p= \frac 1 {\Vol(L)}\int_L h_Q(v^t.x)^p dx.
  \label{eq:cen_hi}
\end{equation}
\end{definition}
The $L^p$ centroid body operator, introduced by Lutwak and Zhang \cite{LZ97}, is then, up to different normalization constants, $\Gamma_p = \Gamma_{[-1,1],p}$. The classical centroid body operator is then $\Gamma=\Gamma_1$. One can readily verify that, for $T\in GL_n(\R)$ and a compact set $L\subset \M$ with positive volume, one has
    \begin{equation}
        \G  (T. L) = T.(\G  L). 
        \label{trans_cent}
    \end{equation}
By sending $p\to\infty,$ we have, for a compact domain $L\subset\M$ with positive volume, that $\Gamma_{Q,\infty} L$ is a convex body in $\M[n,1]$ defined via the support function
\begin{equation}
\label{eq:higher_order_L}
h_{\Gamma_{Q,\infty} L}(\xi)=\max_{x\in L}h_Q(\xi^t.x). \end{equation}

Let $\rho_M(u)=\sup\{r>0:ru\in M\}$ be the radial function of a compact set $M\subset\M[n,1]$ containing the origin in its interior, which is well-defined and continuous as a function on $\s$ when, e.g., $M$ is also convex. Then, for $K\in\conbod[n,1]$, its $q$th higher-order radial mean body $R^m_q K$ is given by, for $q>0$,
\begin{equation}\rho_{R^m_q K}^q (u)=\frac{q}{\vol(K)}\int_0^\infty g_{K,m}(tu)t^{q-1}dt = \frac{q}{\vol(K)}\int_0^{\rho_{D^m(K)}(u)} g_{K,m}(tu)t^{q-1}dt.
\label{eq:layer}
\end{equation}
Here, $$g_{K,m}(x)=\vol\left(K\cap_{i=1}^m (K+x_i)\right)$$ is the higher-order covariogram of $K$ and $D^m(K)=\text{supp}(g_{K,m})$ is the $m$th higher order difference body of $K$. From continuity, $R^m_\infty K = D^m(K)$. One can use Fubini's theorem to write, with $\theta=(\theta_1,\dots,\theta_m)\in \S\subset\M$,
\begin{equation}
\label{eq:radial_og}
    \rho_{R_q^{m} K}(\theta)^q=\frac{1}{\vol(K)} \int_K \left(\min _{i=1 \ldots m}\left\{\rho_{K-x}(-\theta_i)\right\}\right)^q d x.
\end{equation}
Using \eqref{eq:radial_og}, the bodies $R^m_q K$ also exist for $q\in (-1,0]$, with $q=0$ following from continuity. There is a ``layer cake" formula analogous to \eqref{eq:layer} for $q\in (-1,0)$, but we will not make use of it in this work. Notation-wise, $R_q K:=R^1_q K$. Schneider was the first \cite{Sch70} to introduce $D^m(K)$, where he established the higher-order Rogers-Shephard inequality
$$\frac{\Vol(D^m(K))}{\vol(K)^m}\leq \binom{nm+n}{n},$$
with equality if, and only if, $K$ is an $n$-dimensional simplex. This, of course, extends the $m=1$ case by Rogers and Shephard \cite{RS57}. When $m=1$ and $n$ is arbitrary, or $m\geq 2$ and $n=1,2$, the affine-invariant quantity $\frac{\Vol(D^m(K))}{\vol(K)^m}$ is minimzed for any symmetric $K$. For $m\geq 2, n\geq 3$, \textit{Schneider's conjecture} is that the quantity is minimized by ellipsoids.

We say a compact set $L$ in $\M[n,1]$ is a star body if it contains the origin, has a continuous radial function as a function on $\s$, and $[0,x]\subset L$ for every $x\in L$. Notice that $c\rho_L=\rho_{cL}$ for $c >0$. This yields that $R_q^{m} cK=c R_q^{m} K$ from \eqref{eq:radial_og}. It was shown in \cite{HLPRY23} that the radial mean bodies satisfy the following pair of set inclusions:  fix $n,m\in\N$ and $Q\in\conbodo[1,m]$. Then, for every $K\in\conbod[n,1]$ and $-1<p\leq q <\infty$ one has the following strict set inclusions
\begin{equation}
\label{eq:radial_jensen}
R^m_pK \subset R^m_qK \subset D^m(K).\end{equation}
Additionally,
\begin{equation}D^m (K) \subseteq {\binom{q+n}{n}}^{\frac{1}{q}} R^m_{q}K \subseteq {\binom{p+n}{n}}^{\frac{1}{p}} R^m_{p} K\subseteq n\vol(K)\Pi^{\circ,m} K,
\label{eq:higher_simplex}
\end{equation}
with equality in any set inclusion if, and only if, $K$ is an $n$-dimensional simplex. The case $m=1$ was originally shown by Gardner and Zhang \cite{GZ98}; in fact, they first introduced the radial mean bodies $R_q K$ in the same work. Here $\Pi^{\circ,m} K = (\Pi^{m} K)^\circ$ and $\Pi^m K$ is the convex body in $\M[n,m]$ given by the following definition.
\begin{definition} \label{d:generalprojectionbody} Let $p \geq 1$, $m \in \N$, and fix $Q \in \conbodo[1,m]$. Given $K \in \conbodo[n,1]$ with the property that $\sigma_{K,p}$ is a finite Borel measure on $\s$ (e.g. $K\in\conbodio[n,1]$), we define the $(L^p, Q)$-projection body of $K$, $\P K $ via the support function 
\[
h_{\P K}(x)^p = \int_{\s} h_Q(u^t.x)^p d\sigma_{K,p}(u)
\]
for $x\in\M$.
\end{definition}
The body $\Pi^m K$ is then $\Pi^m K=\Pi_{-{\Delta_m},1} K$, where $\Delta_m$ is the orthogonal simplex, i.e. the convex hull of $\{o,e_1,\dots,e_m\}$ in $\M[1,m]$. Here, $d\sigma_{K,p}(u)=h_K(u)^{1-p}d\sigma_K(u)$ and $d\sigma_K(u)$ is surface area measure of $K$, i.e. the pushforward of the $(n-1)$ dimensional Hausdorff measure from $\partial K$, the boundary of $K$, to the unit sphere $\s$ under the Gauss map $n_K$.

In fact, the operator $\G$ can be extended to compact domains, in which case the following theorem was proven \cite[Theorem 1.6]{HLPRY23_2}.

\begin{theorem}
    \label{t:NewLpBPC} Fix $n,m \in \N$. 
Let $L\subset \M[n,m]$ be a compact domain with positive volume, $Q \in \conbodo[1,m]$, and $p \geq 1$. Then
    \begin{equation}\label{eq:NewLpBPC}
     \frac{\vol(\G L)}{\Vol(L)^{1/m}} \geq \frac{\vol(\G \PP \B)}{\Vol(\PP \B)^{1/m}}.
\end{equation}
If $L$ is a star body or a compact domain with piecewise smooth boundary, then there is equality if, and only if, $L = \PP E$ up to a set of zero volume for some origin symmetric ellipsoid $E \in \conbodio[n,1]$.
\end{theorem}
We remark that, by setting $Q=[-1,1]$, Theorem~\ref{t:NewLpBPC} yields the classical result from Busemann and Petty \cite{Busemann53,petty61_1} when $p=1$ and the $L^p$ analogue proven by Lutwak, Yang and Zhang \cite{LYZ00} when $p>1.$ By setting $Q = [-\alpha_1, \alpha_2]$, Theorem~\ref{t:NewLpBPC} becomes the asymmetric $L^p$ case by Haberl and Schuster \cite{HS09}. 

\begin{definition}
    Let $n,m \in \N$, $K\in\conbod[n,1]$ and $Q\in\conbodo[1,m]$. Then, $Z^m_p (K,Q)$, the $(L^p,Q)$ mean zonoid of $K$ with respect to $Q$, is given by
    \begin{equation}
Z^m_p (K,Q) = \left(\frac{\Vol(R^m_{nm+p} K)}{\vol(K)^m}\right)^\frac{1}{p} \G R_{nm+p}^m K.
\label{eq:def}
\end{equation} 
\end{definition}
We remark that $R_q^m K$ existing for $q\in (-1,\infty)$ suggests that there is possibility of $Z^m_p (K,Q)$ being well defined for $p \geq -1-nm$. However, this would also require analyzing $\G$ for negative $p$, which we save for a later work. Therefore, we will consider only $p>0$. Let's derive a formula for the support function of these convex bodies. We obtain from polar coordinates and the homogeneity of support functions that
\begin{align*}
    h^p_{\G L}(\theta)&=\frac{1}{\Vol(L)}\int_{\S}\int_0^{\rho_{L}(u)}h_Q(\theta^t.u)^pt^{nm+p-1}dtdu
    \\
    &=\frac{1}{nm+p}\frac{1}{\Vol(L)}\int_{\S}h_Q(\theta^t.u)^p\rho_L(u)^{nm+p}du.
\end{align*}
We then set $L= R^m_{nm+p} K$ and obtain
\begin{align*}
    h^p_{\G R^m_{nm+p} K}(\theta)&=\frac{1}{\Vol(R^m_{nm+p} K)\vol(K)}\int_{\S}h_Q(\theta^t.u)^p\int_0^{\rho_{D^m(K)}(u)} g_{K,m}(tu)t^{nm+p-1}dt du
    \\
    &=\frac{1}{\Vol(R^m_{nm+p} K)\vol(K)}\int_{\S}\int_0^{\rho_{D^m(K)}(u)}h_Q(\theta^t. tu)^p g_{K,m}(tu)t^{nm-1}dt du
    \\
    &=\frac{1}{\Vol(R^m_{nm+p} K)\vol(K)}\int_{D^m(K)}g_{K,m}(x)h_Q(\theta^t.x)^p  dx.
\end{align*}
Consequently, we obtain
\begin{equation}
\label{eq:pretty}
    h^p_{Z^m_p (K,Q)}(\theta)=\frac{1}{\vol(K)^{m+1}} \int_{D^m(K)}g_{K,m}(x)h_Q(\theta^t.x)^p  dx.
\end{equation}
However, notice we can write this as, using Fubini's theorem,
\begin{align*}
    h^p_{Z^m_p (K,Q)}(\theta)&=\frac{1}{\vol(K)^{m+1}} \int_{\M}\int_K\prod_{i=1}^n\chi_{K+x_i}(y)dyh_Q(\theta^t.x)^p  dx
    \\
    &=\frac{1}{\vol(K)^{m+1}} \int_K\int_{\M[n,1]}\chi_{K}(y-x_1)\cdots\int_{\M}\chi_{K}(y-x_m)h_Q(\theta^t.x)^pdx_m\cdots dx_1 dy.
\end{align*}
Then, setting $z_i=y-x_i$ and $y=z_0$ for aesthetics yields
\begin{equation}
\label{eq:final_form}
\begin{split}
    h^p_{Z^m_p (K,Q)}(\theta)&=\frac{1}{\vol(K)^{m+1}} \int_{K}\cdots \int_{K} h_Q(\theta^t.\{z_0-z_i\}_{i=1}^m)^pdz_m\cdots dz_1 dz_0
    \\
    &=\frac{1}{\vol(K)^{m+1}}  \int_{K}\cdots \int_{K} h_{Q^t}(\left(\{z_0-z_i\}_{i=1}^m\right)^t.\theta)^pdz_m\cdots dz_1 dz_0,
    \end{split}
\end{equation}
which is the $(L^p,Q)$ higher-order analogue of \eqref{eq:zhang}.
In \eqref{eq:final_form}, $\{z_0-z_i\}_{i=1}^m$ is the $n\times m$ matrix whose $i$th column is the vector $z_0-z_i$, and we used the fact that $h_A(x) = h_{A^t}(x^t)$; note that $A^t=\{x^t:x\in A\}$.

Our first step is to establish how $Z_p^m(K,Q)$ behaves under affine transformations.
\begin{proposition}
\label{p:affine}
    Fix $p\geq 1$ and $Q\in\conbodo[1,m]$. Then, for every $K\in\conbod[n,1]$ and $A\in GL_n(\R)$, one has
    $$Z_p^m(A.K,Q)=A.Z^m_p (K,Q).$$
\end{proposition}
\begin{proof}
    From \eqref{eq:final_form}, we can write
\begin{align*}
    h^p_{Z^m_p (A.K,Q)}(\theta)    &=\frac{1}{\vol(AK)^{m+1}}\int_{AK}\cdots\int_{AK}h_{Q}(\theta^t.\left(\{z_0-z_i\}_{i=1}^m\right))^pdz_m\cdots dz_1 dz_0
    \\
    &=\frac{1}{\vol(K)^{m+1}}\int_{K}\cdots\int_{K}h_{Q}(\theta^t.\left(\{A.(z_0-z_i)\}_{i=1}^m\right))^pdz_m\cdots dz_1 dz_0
    \\
    &=\frac{1}{\vol(K)^{m+1}}\int_{K}\cdots\int_{K}h_{Q}(\theta^t.A.\left(\{z_0-z_i\}_{i=1}^m\right))^pdz_m\cdots dz_1 dz_0
    \\
    &=\frac{1}{\vol(K)^{m+1}}\int_{K}\cdots\int_{K}h_{Q}((A^t.\theta)^t.\left(\{z_0-z_i\}_{i=1}^m\right))^pdz_m\cdots dz_1 dz_0
    \\
    &=h^p_{Z^m_p (K,Q)}(A^t.\theta)=h^p_{A.Z^m_p (K,Q)}(\theta).
    \end{align*}
\end{proof}
A special case of Proposition~\ref{p:affine} is that, for $t>0$, one has $Z_p^m(tK,Q)=tZ^m_p (K,Q).$ We obtain via Jensen's inequality and \eqref{eq:pretty} that, for $0 < p \leq q$,
\begin{equation}
    Z^m_p (K,Q) \subset Z^m_q (K,Q) \subset \Gamma_{Q,\infty} D^m (K).
\end{equation}
On the other-hand, it is clear that $\G$, as an operator, is monotonic with respect to set-inclusion. Thus, for every fixed $p\in(0,\infty)$, we  have from \eqref{eq:radial_jensen}
\begin{equation}
    \label{eq:symmetric}
    Z^m_p (K,Q) \subseteq \left(\frac{\Vol(R^m_{nm+p} K)}{\vol(K)^m}\right)^\frac{1}{p} \G D^m(K) \subseteq \left(\frac{\Vol(D^m(K))}{\vol(K)^m}\right)^\frac{1}{p} \G D^m(K),
\end{equation}
and, from \eqref{eq:higher_simplex},

\begin{equation}
\label{eq:simplex_2}
\begin{split}
\G D^m (K) &\subseteq {\binom{n(m+1)+p}{nm+p}}^{\frac{1}{nm+p}}  Z^m_p (K,Q) 
\\
&\subseteq \left(\frac{\Vol(R^m_{nm+p} K)}{\vol(K)^m}\right)^\frac{1}{p} n\vol(K)\G\Pi^{\circ,m} K.
\end{split}
\end{equation}
We can then use the set-inclusion \eqref{eq:higher_simplex} again on $\Vol(R^m_{nm+p} K)$ to remove the instance of the radial mean bodies on one side of the inclusion \eqref{eq:simplex_2} and obtain
\begin{equation}
\begin{split}
Z^m_p (K,Q) \subseteq [n\vol(K)]\left[\frac{\vol(K)^{m(n-1)}\Vol(\Pi^{\circ,m} K)}{\frac{1}{n^{nm}}\binom{n(m+1)+p}{nm+p}}\right]^\frac{1}{p}\G\Pi^{\circ,m} K.
\end{split}
 \label{eq:simplex_3}
\end{equation}
There is equality in \eqref{eq:simplex_2} and \eqref{eq:simplex_3} if, and only if, $K$ is an $n$-dimensional simplex.

We next establish continuity of $Z_p^m(\cdot,Q)$ as an operator on $\conbod[n,1]$.
\begin{proposition}\label{prop-conv}
    Fix $p\geq 1$ and $Q\in\conbodo[1,m]$. Let $K\in\conbod[n,1]$ and take a sequence $\{K_i\}\subset \conbod[n,1]$ such that $K_i\to K$ in the Hausdorff metric. Then, $Z_p^m(K_i,Q)\to Z_p^m(K,Q)$ in the Hausdorff metric.
\end{proposition}
\begin{proof}
    For every $\theta\in\s$, one has
    \begin{align*}
        &|\vol(K)^{m+1}h^p_{Z^m_p (K,Q)}(\theta)-\vol(K_i)^{m+1}h^p_{Z^m_p (K_i,Q)}(\theta)|= 
        \\
        &\bigg|\int_{\M[n,1]}\chi_K(z_0)\cdots\int_{\M[n,1]}\chi_K(z_m)h_{Q}(\theta^t.\left(\{z_0-z_i\}_{i=1}^m\right))^pdz_m\cdots dz_1 dz_0 
        \\
        &- \int_{\M[n,1]}\chi_{K_i}(z_0)\cdots\int_{\M[n,1]}\chi_{K_i}(z_m)h_{Q}(\theta^t.\left(\{z_0-z_i\}_{i=1}^m\right))^pdz_m\cdots dz_1 dz_0\bigg|
        \\
        &\leq \int_{\M[n,1]}\chi_K(z_0)\cdots\int_{\M[n,1]}|\chi_K(z_m)-\chi_{K_i}(z_m)|h_{Q}(\theta^t.\left(\{z_0-z_i\}_{i=1}^m\right))^pdz_m\cdots dz_1 dz_0 
        \\
        &+ \int_{\M[n,1]}|\chi_K(z_0)-\chi_{K_i}(z_0)|\cdots\int_{\M[n,1]}\chi_{K_i}(z_m)h_{Q}(\theta^t.\left(\{z_0-z_i\}_{i=1}^m\right))^pdz_m\cdots dz_1 dz_0,
    \end{align*}
    and the claim then follows, as $|\chi_K(z_j)-\chi_{K_i}(z_j)|$ is bounded by $2\chi_{K_i\triangle K}(z_j)$ for all $j\in \{0,\dots,m\}$ and all $i$, and $h_Q(\theta^t.\cdot)$ can be taken to be uniformly bounded on $K\cup_{i} K_i$.
\end{proof}

\section{Main Theorem}

This section is dedicated to proving the following theorem; we recall that a convex body is said to be \textit{strictly convex} if its boundary does not contain a line segment.
\begin{theorem}
\label{t:main}
Fix $p\geq 1$, $m,n\in\N$ and $Q\in\conbodo[1,m]$. For $K\in\conbod[n,1]$, one has
$$\frac{\vol(Z_p^m(K,Q))}{\vol(K)} \geq \frac{\vol(Z_p^m(\B,Q))}{\vol(\B)}.$$
There is equality if $K$ is an ellipsoid. In fact, if $Q$ is strictly convex, then equality implies $K$ is an ellipsoid.
\end{theorem}

We recall \textit{Steiner symmetrization}: given $u\in\s$, the {\it Steiner symmetral} $S_uK$ of $K$ about $u^\perp$ constructs a rearrangement of $K$ that is symmetric about $u^\perp$. Specifically, write $K_u$ for the orthogonal projection of $K$ onto $u^\perp$. Then, we can write,
$$K=\left\{y^{\prime}+t u:-\underline{\ell}_u\left(K ; y^{\prime}\right) \leq t \leq \bar{\ell}_u\left(K ; y^{\prime}\right) \text { for } y^{\prime} \in K_u\right\},$$
where $\bar{\ell}_u\left(K ; y^{\prime}\right): K_u \rightarrow \mathbb{R}$ and $\underline{\ell}_u\left(K ; y^{\prime}\right): K_u \rightarrow \mathbb{R}$ are the {\it overgraph} and {\it undergraph} of $K$ respectively. If one sets, for $y\in K_u$, $$m_{y^{\prime}}(u)=m_{y^{\prime}}(K;u)=\frac{1}{2}\left[\bar{\ell}_u\left(K ; y^{\prime}\right)-\underline{\ell}_u\left(K ; y^{\prime}\right)\right],$$
then the midpoint of $K\cap (u\R + y^\prime)$ is $y^\prime + m_{y^{\prime}}(u)u$; it will be convenient to write $\text{vol}_{1}(K\cap (u\R + y^\prime))=\sigma_{y^\prime}(u)=\sigma_{y^\prime}$. The Steiner symmetral $S_u K$ is then constructed by moving $m_{y^{\prime}}(u)$ to the origin in $u\R$; in particular, this yields $(S_u K)_u = K_u$, $$\underline{\ell}_u\left(S_u K ; y^{\prime}\right)=\bar{\ell}_u\left(S_u K ; y^{\prime}\right)=\frac{1}{2}\left[\bar{\ell}_u\left(K ; y^{\prime}\right)+\underline{\ell}_u\left(K ; y^{\prime}\right)\right],$$ and, from the translation invariance of the Lebesgue measure, 
\begin{equation}
\label{eq:steiner}
\vol(S_u K)=\vol(K).
\end{equation}

The first step to proving Theorem~\ref{t:main} is to relate $Z_p^m(S_u K,Q)$ with $Z_p^m(K,Q)$. It will be convenient to write, for $z\in\M[n,1]$ and $u\in\s$, $z=(z^\prime,t)$, where $z^\prime\in u^\perp$ and $z=z^\prime + tu$. Of course, we are suppressing in this notation the dependence on $u$.
\begin{lemma}
\label{l:set}
    Fix $p \geq 1$, $m,n\in\N$ and $Q\in\conbodo[1,m]$. Then, for $u\in\s$ and $y_1^\prime,y_2^\prime\in u^\perp$, one has
    \begin{equation}h_{Z_p^m(S_u K,Q)}\left( \frac{y_1^\prime + y_2^\prime}{2},\pm 1\right) \leq \frac{h_{Z_p^m(K,Q)}(y_1^\prime, 1)+h_{Z_p^m(K,Q)}(y_2^\prime, -1)}{2}.
    \label{eq:desired}
    \end{equation}
    Suppose $Q$ is strictly convex. Then, there is equality if, and only if, all the midpoints of the chords of $K$ parallel to $u$ lie in a hyperplane.
\end{lemma}
\begin{proof}
    From \eqref{eq:final_form}, we can write, by decomposing $z_i=z_i^\prime + t_iu$ and using Fubini's theorem
    \begin{align*}
        h^p_{Z^m_p (K,Q)}(y^\prime,\pm 1)&=\frac{1}{\vol(K)^{m+1}} \int_{K}\int_{K}\cdots\int_{K}h_Q((y^\prime,\pm 1)^t.\{z_0-z_i\}_{i=1}^m)^pdz_m\cdots dz_1 dz_0
        \\
        &=\frac{1}{\vol(K)^{m+1}} \int_{K_u}\int_{K_u}\cdots\int_{K_u}\int_{m_{z_0^\prime}-\frac{1}{2}\sigma_{z^\prime_0}}^{m_{z_0^\prime}+\frac{1}{2}\sigma_{z^\prime_0}}\cdots \int_{m_{z_m^\prime}-\frac{1}{2}\sigma_{z^\prime_m}}^{m_{z_m^\prime}+\frac{1}{2}\sigma_{z^\prime_m}} 
        \\
        &\quad\quad h_Q((y^\prime,\pm 1)^t.\{(z_0^\prime-z_i^\prime,t_0-t_i)\}_{i=1}^m)^p dt_m\cdots dt_1 dt_0 dz^\prime_m\cdots dz^\prime_1 dz^\prime_0
        \\
        &=\frac{1}{\vol(K)^{m+1}} \int_{K_u}\int_{K_u}\cdots\int_{K_u}\int_{m_{z_0^\prime}-\frac{1}{2}\sigma_{z^\prime_0}}^{m_{z_0^\prime}+\frac{1}{2}\sigma_{z^\prime_0}}\cdots \int_{m_{z_m^\prime}-\frac{1}{2}\sigma_{z^\prime_m}}^{m_{z_m^\prime}+\frac{1}{2}\sigma_{z^\prime_m}} 
        \\
        &\quad\quad h_Q(\{(y^\prime.(z_0^\prime-z_i^\prime) \pm (t_0-t_i))\}_{i=1}^m)^p dt_m\cdots dt_1 dt_0 dz^\prime_m\cdots dz^\prime_1 dz^\prime_0.
    \end{align*}

    By performing a variable substitution $s_i=\pm t_i \mp m_{z^\prime_i}$, this becomes
     \begin{align*}
        &h^p_{Z^m_p (K,Q)}(y^\prime,\pm 1)
        \\
        &=\frac{1}{\vol(K)^{m+1}} \int_{K_u}\int_{K_u}\cdots\int_{K_u}\int_{-\frac{1}{2}\sigma_{z^\prime_0}}^{\frac{1}{2}\sigma_{z^\prime_0}}\cdots \int_{-\frac{1}{2}\sigma_{z^\prime_m}}^{\frac{1}{2}\sigma_{z^\prime_m}} 
        \\
        &\quad\quad h_Q(\{(y^\prime.(z_0^\prime-z_i^\prime)+(s_0-s_i)\pm (m_{z_0^\prime} - m_{z_i^\prime}))\}_{i=1}^m)^p ds_m\cdots ds_1 ds_0 dz^\prime_m\cdots dz^\prime_1 dz^\prime_0.
    \end{align*}
    Using that $(S_u K)_u=K_u$ and Fubini's theorem again, we then obtain
    \begin{equation}
    \label{eq:steiner_cord}
    \begin{split}
        &h^p_{Z^m_p (K,Q)}(y^\prime,\pm 1)=\frac{1}{\vol(K)^{m+1}} \int_{S_u K}\int_{S_u K}\cdots\int_{S_u K}
        \\
        &\quad h_Q(\{(y^\prime.(z_0^\prime-z_i^\prime)+(s_0-s_i)\pm (m_{z_0^\prime} - m_{z_i^\prime}))\}_{i=1}^m)^p ds_mdz^\prime_m\cdots ds_1dz^\prime_1 ds_0dz^\prime_0 .
    \end{split}
    \end{equation}
    We next replace $K$ with $S_u K$. Recalling that $m_{z_i^\prime}$ is a function of $K$, and by definition $m_{z_i^\prime}=0$ for $S_u K$, we obtain
    \begin{align*}
        &h^p_{Z^m_p (S_u K,Q)}(y^\prime,\pm 1)=\frac{1}{\vol(K)^{m+1}} \int_{S_u K}\int_{S_u K}\cdots\int_{S_u K}
        \\
        &\quad h_Q(\{(y^\prime.(z_0^\prime-z_i^\prime)+(s_0-s_i))\}_{i=1}^m)^p ds_mdz^\prime_m\cdots ds_1dz^\prime_1 ds_0dz^\prime_0 .
    \end{align*}
    We then pick a $y^\prime$ of the form $\frac{1}{2}(y_1^\prime + y_2^\prime)$ to obtain
    \begin{align*}
        &2^ph^p_{Z^m_p (S_u K,Q)}\left(\frac{y_1^\prime + y_2^\prime}{2},\pm 1\right)=\frac{2^p}{\vol(K)^{m+1}} \int_{S_u K}\int_{S_u K}\cdots\int_{S_u K}
        \\
        &\quad h_Q\left(\left\{\left(\frac{y_1^\prime + y_2^\prime}{2}.(z_0^\prime-z_i^\prime)+(s_0-s_i)\right)\right\}_{i=1}^m\right)^p ds_mdz^\prime_m\cdots ds_1dz^\prime_1 ds_0dz^\prime_0 
        \\
        &=\frac{1}{\vol(K)^{m+1}} \int_{S_u K}\int_{S_u K}\cdots\int_{S_u K}
        \\
        &\quad h_Q(\{((y_1^\prime + y_2^\prime).(z_0^\prime-z_i^\prime)+2(s_0-s_i))\}_{i=1}^m)^p ds_mdz^\prime_m\cdots ds_1dz^\prime_1 ds_0dz^\prime_0
        \\
        &=\frac{1}{\vol(K)^{m+1}} \int_{S_u K}\int_{S_u K}\cdots\int_{S_u K}
        \\
        &\quad h_Q(\{((y_1^\prime + y_2^\prime).(z_0^\prime-z_i^\prime)+[(s_0-s_i) +(m_{z_0^\prime} - m_{z_i^\prime})] + [(s_0-s_i) -(m_{z_0^\prime} - m_{z_i^\prime})])\}_{i=1}^m)^p 
        \\
        &\quad\quad \times ds_mdz^\prime_m\cdots ds_1dz^\prime_1 ds_0dz^\prime_0 .
    \end{align*}
    From the 1-homogeneity and convexity of support functions, they are sublinear. Thus, we can continue:
    \begin{align*}
        &2^ph^p_{Z^m_p (S_u K,Q)}\left(\frac{y_1^\prime + y_2^\prime}{2},\pm 1\right)\leq \frac{1}{\vol(K)^{m+1}} \int_{S_u K}\int_{S_u K}\cdots\int_{S_u K}
        \\
        & \left(h_Q(\{(y_1^\prime.(z_0^\prime-z_i^\prime) +[(s_0-s_i) +(m_{z_0^\prime} - m_{z_i^\prime})] )\}_{i=1}^m)\right.
        \\
        &\quad\left.+h_Q(\{(y_2^\prime.(z_0^\prime-z_i^\prime)+ [(s_0-s_i) -(m_{z_0^\prime} - m_{z_i^\prime})])\}_{i=1}^m)\right)^p ds_mdz^\prime_m\cdots ds_1dz^\prime_1 ds_0dz^\prime_0 .
    \end{align*}
    We then take $p$th root of both sides and use Minkowski's integral inequality and \eqref{eq:steiner_cord} to obtain the desired inequality \eqref{eq:desired}. For the equality conditions, we must have equality in the sub-linearity of $h_Q$ and the use of Minkowski's integral inequality.

Suppose $Q$ is strictly convex. Then, both instances of equality imply there exists $\lambda \geq 0$ such that, for all $i=1,\dots,m$, one has
$$y_1^\prime.(z_0^\prime-z_i^\prime)+[(s_0-s_i) +(m_{z_0^\prime} - m_{z_i^\prime})]=\lambda (y_2^\prime.(z_0^\prime-z_i^\prime)+ [(s_0-s_i) -(m_{z_0^\prime} - m_{z_i^\prime})])$$
for fixed, but arbitrary $(z_0^\prime,s_0)\in K$ and all $(z_i^\prime,s_i)\in K$. Upon re-arrangment, this can be written as
\begin{equation}(y_1^\prime-\lambda y_2^\prime).(z_0^\prime-z_i^\prime)+(1+\lambda)(m_{z_0^\prime} - m_{z_i^\prime})=(\lambda-1) (s_0-s_i).
\label{eq:almost_plane}
\end{equation}
We now determine the value of $\lambda$. Notice that the left-hand-side of \eqref{eq:almost_plane} is independent of the $s_i$. In particular, from the fact that $K$ has non-empty interior, we can then vary the $z_i$ while keeping the $z^\prime_i$ fixed, so that the new $(z^\prime_i,s_i)$ are still in $K$. Consequently, \eqref{eq:almost_plane} still holding implies $\lambda=1$. With this choice of $\lambda$, \eqref{eq:almost_plane} reduces to the formula for a hyperplane, and the claim follows.
\end{proof}

Having related the support functions of $Z^m_p (K,Q)$ and $Z^m_p (S_u K,Q)$, we will show that $Z^m_p (S_u K,Q) \subseteq S_u Z^m_p (K,Q)$. This will be done by using the obvious fact that $K\subseteq L$ for two convex bodies $K$ and $L$ if, and only if, for every $u\in\s$, $\bar{\ell}_u\left(K ; \cdot\right)$ and $\underline{\ell}_u\left(K ; \cdot \right)$ are point-wise smaller than $\bar{\ell}_u\left(L ; \cdot \right)$ and $\underline{\ell}_u\left(L ; \cdot \right)$ respectively on $K_u$. We will need the following fact from \cite[Lemma 1.2]{LYZ10_2}; the relative interior of a set $L\subset \M[n,1]$ is denoted by $\text{relint}(L)$.

\begin{proposition}
\label{p:graphs}
    Let $K\in\conbodo[n,1]$. Then, for $u\in\s$, the overgraph and undergraph functions of $K$ in the direction $u$ at the point $y^\prime \in \text{relint} (K_u)$ are given by
    \begin{align}
        \bar{\ell}_u\left(K ; y^\prime\right) &=\min_{x^\prime \in u^\perp} \left\{h_K(x^\prime,1)-\langle x^\prime,y^\prime \rangle  \right\}
        \\
        \underline{\ell}_u\left(K ; y^\prime\right) &=\min_{x^\prime \in u^\perp} \left\{h_K(x^\prime,-1)-\langle x^\prime,y^\prime \rangle  \right\}.
    \end{align}
\end{proposition}

\begin{proposition}
\label{p:steiner}
    Fix $p \geq 1$, $m,n\in\N$ and $Q\in\conbodo[1,m]$. Then, for every $u\in\s$,
    $$Z^m_p (S_u K,Q) \subseteq S_u Z^m_p (K,Q).$$
      Suppose $Q$ is strictly convex. Then, there is equality if, and only if, all the midpoints of the chords of $K$ parallel to $u$ lie in a hyperplane.
\end{proposition}
\begin{proof}
    Fix $y^\prime\in \text{relint}(Z^m_p (K,Q)_u)$. From Proposition~\ref{p:graphs}, there exist $z_1^\prime=z_1^\prime(y^\prime),z_2^\prime=z_2^\prime(y^\prime)\in u^\perp$ such that
\begin{align}
        \bar{\ell}_u\left(Z^m_p (K,Q) ; y^\prime\right) &= h_{Z^m_p (K,Q)}(z_1^\prime,1)-\langle z_1^\prime,y^\prime \rangle,
        \\
        \underline{\ell}_u\left(Z^m_p (K,Q) ; y^\prime\right) &= h_{Z^m_p (K,Q)}(z_2^\prime,-1)-\langle z_2^\prime,y^\prime \rangle.
    \end{align}
    Then, from Lemma~\ref{l:set},
    \begin{align*}
        \bar{\ell}_u\left(S_u (Z^m_p (K,Q)); y^\prime\right) &= \frac{1}{2}\left[\bar{\ell}_u\left(Z^m_p (K,Q); y^\prime\right)+\underline{\ell}_u\left(Z^m_p (K,Q); y^\prime\right)\right]
        \\
        &=\frac{1}{2}\left[h_{Z^m_p (K,Q)}(z_1^\prime,1) + h_{Z^m_p (K,Q)}(z_2^\prime,-1)\right]-\left\langle\left( \frac{z_1^\prime+z_2^\prime}{2}\right),y^\prime \right\rangle
        \\
        &\geq h_{Z_p^m(S_u K,Q)}\left( \frac{z_1^\prime + z_2^\prime}{2}, 1\right) -\left\langle\left( \frac{z_1^\prime+z_2^\prime}{2}\right),y^\prime \right\rangle
        \\
        &\geq \min_{x^\prime \in u^\perp}\left\{h_{Z^m_p (S_u K,Q)}(x^\prime,1)-\left\langle x^\prime,y^\prime \right\rangle\right\}
        \\
        &=\bar{\ell}_u\left(Z^m_p (S_u K,Q) ; y^\prime\right)
    \end{align*}
    and
    \begin{align*}
        \underline{\ell}_u\left(S_u (Z^m_p (K,Q)); y^\prime\right) &= \frac{1}{2}\left[\bar{\ell}_u\left(Z^m_p (K,Q); y^\prime\right)+\underline{\ell}_u\left(Z^m_p (K,Q); y^\prime\right)\right]
        \\
        &=\frac{1}{2}\left[h_{Z^m_p (K,Q)}(z_1^\prime,1) + h_{Z^m_p (K,Q)}(z_2^\prime,-1)\right]-\left\langle\left( \frac{z_1^\prime+z_2^\prime}{2}\right),y^\prime \right\rangle
        \\
        &\geq h_{Z_p^m(S_u K,Q)}\left( \frac{z_1^\prime + z_2^\prime}{2},- 1\right) -\left\langle\left( \frac{z_1^\prime+z_2^\prime}{2}\right),y^\prime \right\rangle
        \\
        &\geq \min_{x^\prime \in u^\perp}\left\{h_{Z^m_p (S_u K,Q)}(x^\prime,-1)-\left\langle x^\prime,y^\prime \right\rangle\right\}
        \\
        &=\underline{\ell}_u\left(Z^m_p (S_u K,Q) ; y^\prime\right).
    \end{align*}

For the equality conditions, notice there is equality if, and only if, there is equality in Lemma~\ref{l:set}.
    
\end{proof}
We next recall the following classical fact about Steiner symmetrization (see e.g. \cite[Theorem 6.6.6]{Web94}).
\begin{lemma}
\label{l:steiner}
    Let $K$ be a convex body in $\R^n$. Then, there exists a sequence of directions $\{u_j\}_{j=1}^\infty\subset \s$ such that, if we define $S_0 K=K$, $S_1 K = S_{u_1} K$ and $S_{j}K=S_{u_j} S_{j-1}K,$ then $S_j K \to B_K$ in the Hausdorff metric.
\end{lemma}

The proof of Theorem~\ref{t:main} is now immediate.
\begin{proof}[Proof of Theorem~\ref{t:main}] Let $\{u_j\}$ be the sequence of directions from Lemma~\ref{l:steiner}. Then, from \eqref{eq:steiner}, Proposition~\ref{p:steiner}, and Proposition~\ref{prop-conv},
\begin{align*}\vol(Z^m_p (K,Q)) &= \vol(S_{u_1} Z^m_p (K,Q)) \geq \vol(Z^m_p (S_1 K,Q))
\\ 
&= \vol(S_{u_2} Z^m_p (S_ 1K,Q)) \geq \vol(Z^m_p (S_2 K,Q))
\\
&\vdots
\\
&\geq \vol(Z^m_p (B_K,Q)).
\end{align*}
Next, from Proposition~\ref{p:affine}, we can write $Z^m_p (B_K,Q) = Z^m_p (\vol(\B)^{-1/n}\vol(K)^{1/n} \B,Q) = (\frac{\vol(K)}{\vol(\B)})^{\frac{1}{n}} Z^m_p (\B,Q)$.

Finally, for equality conditions, clearly from direct substitution and Proposition~\ref{p:affine}, there is equality for $K$ an ellipsoid. Now, suppose there is equality. Then, there is equality in Proposition~\ref{p:steiner} for every $u$. Thus, if $Q$ is strictly convex, there is equality if, and only if, for every $u$, the chords of $K$ parallel to $u$ have midpoints lying in a hyperplane, which characterizes an ellipsoid.
\end{proof}
Notice that we can use \eqref{eq:def} to write the result from Theorem~\ref{t:main} as
$$\left(\frac{\Vol(R^m_{nm+p} K)}{\vol(K)^m}\right)^\frac{n}{p}\frac{\vol(\G R^m_{nm+p} K)}{\vol(K)} \geq \left(\frac{\Vol(R^m_{nm+p} \B)}{\vol(\B)^m}\right)^\frac{n}{p}\frac{\vol(\G R^m_{nm+p} \B)}{\vol(\B)}.$$
We can combine this inequality with Theorem~\ref{t:NewLpBPC} to obtain
\begin{align*}
\left(\frac{\Vol(R^m_{nm+p} K)}{\vol(K)^m}\right)^\frac{n}{p}\frac{\vol(\G R^m_{nm+p} K)}{\vol(K)} \geq &\left(\frac{\Vol(R^m_{nm+p} \B)}{\vol(\B)^m}\right)^\frac{n}{p}
\\
&\times\left(\frac{\Vol(R^m_{nm+p} \B)}{\Vol(\PP \B)}\right)^{\frac{1}{m}}\frac{\vol(\G \PP \B)}{\vol(\B)},
\end{align*}
with equality if, and only if when $Q$ is strictly convex, $K$ is an ellipsoid (since $\G \PP E$ is an ellipsoid when $E$ is an an ellipsoid, see \cite[Lemma 3.12]{HLPRY23_2}).

\section{Classification Results}
Recall that a zonoid can be defined as cosine transformation of a Borel measure on a sphere. That is, a convex body $M\subset\R^n$ is a \textit{zonoid} if there exists a Borel measure $\nu$ on $\s$ such that
$$h_{M}(\theta)=\int_{\s}|\langle u,\theta \rangle|d\nu(u).$$ Notable examples include projection bodies $\Pi K$, with $\nu=\frac{\sigma_K}{2}$ (in fact, every zonoid is a translate of a projection body) and mean zonoids $\tilde Z K$ given by \eqref{eq:zhang} (which was verified by Zhang \cite{Zhang91}).

Lutwak, Yang and Zhang, among others, studied \cite{LYZ04_2,LYZ05} $L^p$ zonoids, where $M$ is an $L^p$ zonoid, $p\geq 1$, if there exists a Borel measure $\nu$ on $\S$ such that
$$h_{M}(\theta)^p=\int_{\s}|\langle u,\theta\rangle|^pd\nu(u).$$
Notable examples of $L^p$ zonoids include $L^p$ projection bodies $\Pi_p K$ introduced by Lutwak, Yang and Zhang \cite{LYZ00} (with $\nu$ being, up to a constant, $\sigma_{K,p}$) and the $L^p$ mean zonoids $\tilde Z_p K$ given by \eqref{eq:xgl} (which was verified by Xi, Guo and Leng \cite{XGL14}).

We now introduce the concept of zonoids in the higher-order setting. Interestingly enough, we are in position to define two such types of zonoids, one in $\M[n,m]$ and one in $\M[n,1]$.
\begin{definition}
    Fix $Q\in\conbodo[1,m]$ and $p\geq 1$. Then, we say a convex body $L\in\conbodo[n,m]$ is a higher-order $(L^p,Q)$ zonoid if there exists a Borel measure $\nu$ on $\s$ such that
    $$h_{L}^p(x)=\int_{\s}h_Q(u^t.x)^pd\nu(u).$$
\end{definition}
Clearly, $(L^p,Q)$ projection bodies $\P K$ are higher-order $(L^p,Q)$ zonoids, with $\nu=\sigma_{K,p}$. The bodies $Z^m_p(K,Q)$ do not fit in this framework, since $Z^m_p(K,Q)\in\conbodo[n,1]$.
\begin{definition}
    Fix $Q\in\conbodo[1,m]$ and $p\geq 1$. Then, we say a convex body $M\in\conbodo[n,1]$ is an $(L^p,Q)$ zonoid if there exists a Borel measure $\nu$ on $\S$ such that
    $$h_{M}^p(\theta)=\int_{\S}h_Q(\theta^t.u)^pd\nu(u).$$
\end{definition}
We conclude by showing that $Z^m_p(K,Q)$ are $(L^p,Q)$ zonoids.
\begin{proposition}
    Fix $p\geq 1,m,n\in\N$ and $Q\in\conbodo[1,m]$. For $K\in\conbod[n,1],$ one has that $Z^m_p(K,Q)$ are $(L^p,Q)$ zonoids.
\end{proposition}
\begin{proof}
    This follows from the representation of $Z^m_p(K,Q)$ implied by \eqref{eq:pretty}:
    \begin{align*}
    h^p_{Z^m_p (K,Q)}(\theta)&=\frac{1}{nm+p}\frac{1}{\vol(K)^{m}}\int_{\S}h_Q(\theta^t.u)^p\rho_{R^m_{nm+p} K}(u)^{nm+p}du.
\end{align*}
Setting $d\nu(u) = \frac{\rho_{R^m_{nm+p} K}(u)^{nm+p}du}{(nm+p)\vol(K)^{m}}$ yields the result.
\end{proof}

{\bf Funding:} Dylan Langharst was funded by the U.S.-Israel Binational Science Foundation (BSF) Grant 2018115 and the Fondation Sciences Math\'ematiques de Paris Postdoctoral program. Dongmeng Xi was supported, in part, by the Austrian Science Fund (FWF): 10.55776/P34446, NSFC 12322103 and NSFC 12071277.

{\bf Acknowledgements: } We are thankful to the referee, who gave a very detailed report of the draft of this work which vastly improved the presentation of the material. Our thanks also to the editor.




\bibliographystyle{acm}
\bibliography{references}
\end{document}